\documentclass{article}%
\usepackage{amsmath}
\usepackage[nomarginpar]{geometry}
\usepackage{graphicx}
\usepackage{amsfonts}
\usepackage{amssymb}%
\providecommand{\U}[1]{\protect\rule{.1in}{.1in}}
\newtheorem{theorem}{Theorem}

\newtheorem{lemma}[theorem]{Lemma}

\newenvironment{proof}[1][Proof]{\noindent\textbf{#1.} }{\ \rule{0.5em}{0.5em}}
\graphicspath{{D:/Dropbox/HararyWeight/Slike/}}
\begin{document}

\title{Extremal unicyclic graphs with respect to additively weighted Harary index}
\author{Jelena Sedlar\\University of Split, Faculty of civil engineering, architecture and geodesy, \\Matice hrvatske 15, HR-21000, Split, Croatia.}
\maketitle

\begin{abstract}
In this paper we define cycle-star graph $CS_{k,n-k}$ to be a graph on $n$
vertices consisting of the cycle of length $k$ and $n-k$ leafs appended to the
same vertex of the cycle. Also, we define cycle-path graph $CP_{k,n-k}$ to be
a graph on $n$ vertices consisting of the cycle of length $k$ and of path on
$n-k$ vertices whose one end is linked to a vertex on a cycle. We establish
that cycle-star graph $CS_{3,n-3}$ is the only maximal graph with respect to
additively weighted Harary index among all unicyclic graphs on $n$ vertices,
while cycle-path graph $CP_{3,n-3}$\ is the only minimal unicyclic graph (here
$n$ must be at least 5)$.$ The values of additively weighted Harary index for
extremal unicyclic graphs are established, so these values are the upper and
the lower bound for the value of additively weighted Harary index on the class
of unicyclic graphs on $n$ vertices.

\end{abstract}

%


{\bf AMS Subject Classifcation:} 05C35%

\section{Introduction}

A topological index of a graph is a number attributed to a graph in a way that
it is derived from the structure of the graph but so that it doesn't depend on
the labeling of vertices in a graph. Chemical graph theory is a branch of
graph theory whose focus of interest is finding topological indices of
chemical graphs (i.e. graphs that represent chemical molecules) which
correlate well with chemical properties of the corresponding molecules. One of
the most famous topological indices is Wiener index, defined as the sum of all
distances between different vertices of a molecular graph, introduced by
Wiener in 1947 (see \cite{ref_Wiener(1947)}) in a paper concerned with boiling
points of alkanes. In a research that followed many other useful properties of
Wiener index were discovered (for a survey of mathematical properties and
chemical applications of Wiener index one can look into
\cite{ref_Dobrynin(2001)}, \cite{ref_Dobrynin(2002)}, \cite{ref_Nikolic(1995)}%
). But, contrary to chemical intuition, the contribution of close pairs of
vertices to the overall value of the index was much smaller than that of
distant vertices. To deal with this inconsistency, the new index was proposed
which was named Harary index (\cite{ref_Ivanciuc(1993)},
\cite{ref_Olavsic(1993)}). Harary index is defined as the sum of all
reciprocal values of distances between different vertices of a molecular
graph. Properties of Harary index were then extensively researched (see for
example \cite{ref_SvojstvaIlic}, \cite{ref_SvojstvaWagner},
\cite{ref_HararyUnicyclic}, \cite{ref_Zhou(2008)}). But it turned out that
this modification of Wiener index has not solved the inconsistency as
expected. In order to improve the performance of Harary-type indices, several
modifications were proposed recently. In \cite{ref_Bruckler(2011)} the authors
increased the attenuation of contributions of vertex pairs with their
distance. In \cite{ref_Doslic(2013)} an attempt was made in a different
direction, the authors introduced a correction that gives more weight to the
contributions of pairs of vertices of high degrees. This modification was
named additively weighted Harary index. The aim of this paper is to establish
the upper and lower bound for the value of additively weighted Harary index on
the class of unicyclic graphs and to characterize all extremal graphs.

The present paper is organized as follows. In Section 2 we introduce necessary
definitions and preliminary results. In Section 3 we characterize all maximal
unicyclic graphs by introducing several graph transformations which increase
the value of additively weighted Harary index and which, when applied combined
finitely many times, lead to extremal unicyclic graphs. In Section 4 we do the
same for minimal unicyclic graphs with respect to additively weighted Harary
index. Finally, in Section 5 we give the conclusion and directions for further
research. The paper is completed with acknowledgements and the list of references.

\section{Preliminaries}

Let $G=(V,E)$ be a graph with set of vertices $V$ and set of edges $E.$ In
this paper all graphs are finite and simple. For a pair of vertices $u,v\in V$
the \emph{distance} $d_{G}(u,v)$ is defined as the length of the shortest path
between $u$ and $v.$ The \emph{degree} $\delta_{G}(v)$ of a vertex $v\in V$ is
defined as the number of vertices in $V$ neighboring to $v.$ A \emph{leaf} in
a graph $G$ is every vertex in $G$ of degree $1.$ \emph{Additively weighted
Harary index} is defined as
\[
H_{A}(G)=%
{\displaystyle\sum_{u\not =v}}
\frac{\delta_{G}(u)+\delta_{G}(v)}{d_{G}(u,v)}.
\]
We say that graph $G$ is a \emph{tree} if $G$ doesn't contain a cycle. We say
that $G$ is a \emph{unicyclic} graph if $G$ contains exactly one cycle. The
set of all unicyclic graphs on $n$ vertices will be denoted by $\mathcal{U}%
(n).$ The only cycle in a unicyclic graph $G$ will be denoted by $C$ and the
length of $C$ will be denoted by $c.$ Usually, we will suppose $C$ consists of
path $P=w_{i}w_{i+1}\ldots w_{i+c-1}$ and an edge $w_{i}w_{i+c-1}.$ Starting
index $i$ will not always be the same. Vertex $w_{i-j}$ will denote $j-$th
vertex on $C$ from $w_{j}$ in negative direction, while $w_{i+j}$ will denote
$j-$th vertex on $C$ from $w_{j}$ in positive direction. Note that in this way
we possibly introduce alternative labels of vertices on $C.$ For example, if
$C=w_{0}w_{1}\ldots w_{5},$ then $w_{-2}=w_{4}$ and $w_{7}=w_{1}.$ In other
words, indices are added modulo $c$. We say $v\in G\backslash C$ is
\emph{branching} if $\delta_{G}(v)\geq2,$ we say $v\in C$ is \emph{branching}
if $\delta_{G}(v)\geq3.$

A \emph{cycle-star graph} is a unicyclic graph consisting only of a cycle and
leafs appended to vertices of the cycle will be called cycle-star graph. Note
that a cycle-star graph can be obtained from stars and cycle by identifying
central vertex of stars with different vertices on cycle, hence the name. A
cycle-path graph is a unicyclic graph consisting only of a cycle and at most
one path appended to each vertex of the cycle. We denote with $CS_{k,n-k}$ a
cycle-star graph on $n$ vertices consisting of cycle of length $k$ with $n-k$
leafs appended to the same vertex of the cycle. We denote with $CP_{k,n-k}$ a
cycle-path graph on $n$ vertices consisting of cycle of length $k$ with path
on $n-k$ vertices whose end vertex is linked to a vertex on the cycle. These
notions are illustrated in Figures \ref{Figure01} and \ref{Figure02}.

\begin{figure}[h]
\begin{center}%
\begin{tabular}
[c]{llll}%
a) & $\text{\raisebox{-1\height}{\includegraphics[scale=0.2]{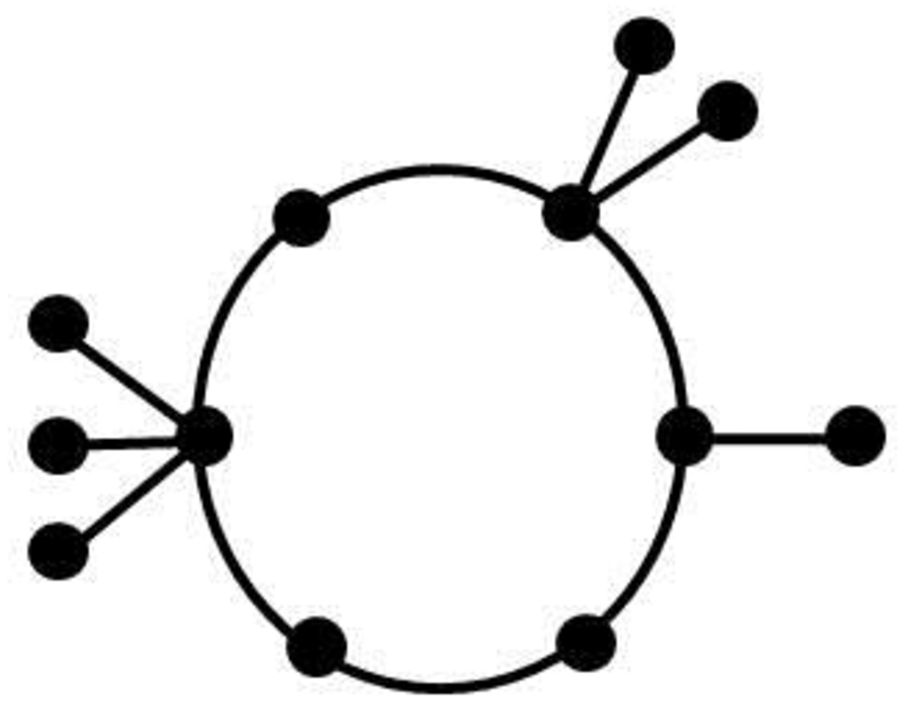}}}$
& b) & $\text{\raisebox{-1\height}{\includegraphics[scale=0.2]{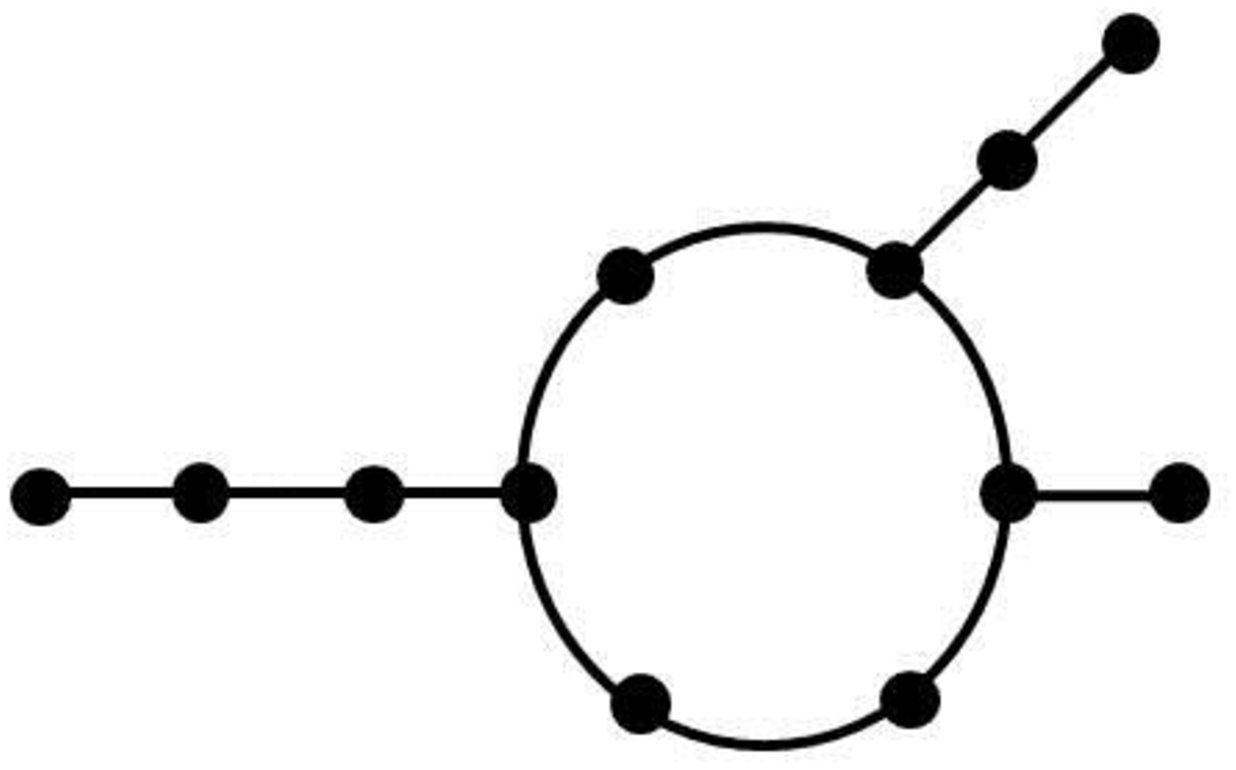}}}%
$%
\end{tabular}
\end{center}
\caption{An example of: a) cycle-star, b) cycle path.}%
\label{Figure01}%
\end{figure}

\begin{figure}[h]
\begin{center}%
\begin{tabular}
[c]{llll}%
a) & $\text{\raisebox{-1\height}{\includegraphics[scale=0.2]{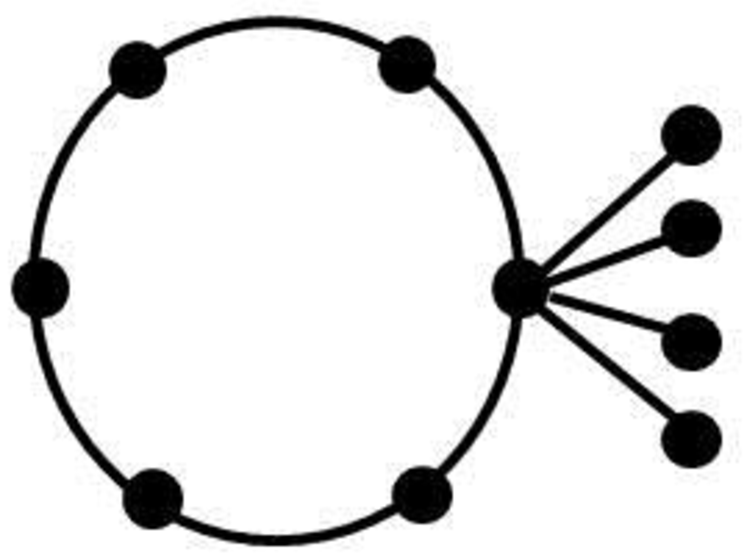}}}$
& b) & $\text{\raisebox{-1\height}{\includegraphics[scale=0.2]{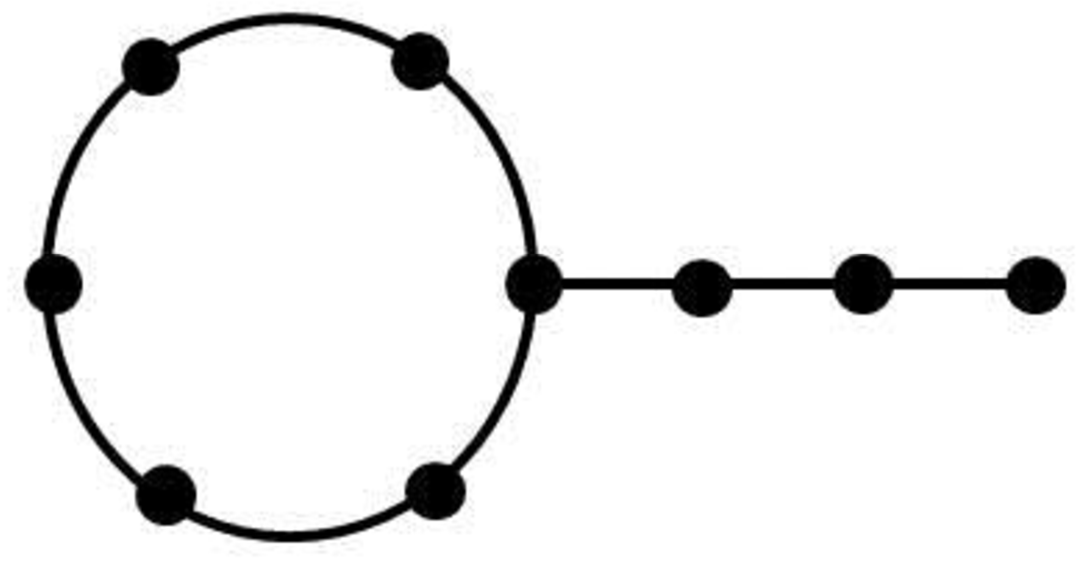}}}%
$%
\end{tabular}
\end{center}
\caption{Graphs: a) $CS_{6.4}$, b) $CP_{6,3}$.}%
\label{Figure02}%
\end{figure}

The cycle on $n$ vertices will be denoted by $C_{n}.$ Obviously, it holds that
$C_{n}\in\mathcal{U}(n).$ It is also convenient to consider that $C_{n}$ is
both a cycle-star and a cycle-path graph, i.e. $C_{n}=CS_{n,0}=CP_{n,0}.$

Finally, we have to introduce numbers which are quite useful when expressing
the values of Harary index. \emph{Harmonic number} $H_{n}$, where $n\geq1$ is
an integer, is defined as $H_{n}=%
{\displaystyle\sum_{i=1}^{n}}
\frac{1}{i}.$

Before we proceed to our main results, we will prove two simple lemmas which
will be useful later.

\begin{lemma}
\label{lema0_valueExtremal}It holds that:

\begin{enumerate}
\item $H_{A}(CS_{3,n-3})=\frac{3}{2}(n^{2}-n+2),$

\item $H_{A}(CP_{3,n-3})=4%
{\displaystyle\sum_{i=1}^{n-2}}
H_{n-i-1}+H_{n-3}+3H_{n-2}+\frac{6n-13}{n-2}.$
\end{enumerate}
\end{lemma}

\begin{proof}
By direct calculation.
\end{proof}

\begin{lemma}
\label{lema0_pomocnaHarmonic}For odd $n\geq5$ it holds that $\frac{3}{2}%
(n^{2}-n+2)-4n\cdot H_{\frac{n-1}{2}}>0.$
\end{lemma}

\begin{proof}
We prove the claim by induction. We introduce notation $E(n)$ for the left
side of inequality. For $n=5$ we have $E(5)=3>0.$ Suppose now that the claim
holds for a $n\geq5,$ then for $n+2$ we have
\begin{align*}
E(n+2)  &  =E(n)+\frac{3}{2}(4n+2)-\frac{8n}{n+1}-\frac{16}{n+1}%
-8H_{\frac{n-1}{2}}\geq\lbrack8n\leq16n]\geq\\
&  \geq E(n)+\frac{3}{2}(4n+2)-\frac{16n}{n+1}-\frac{16}{n+1}-8H_{\frac
{n-1}{2}}=\\
&  =E(n)+6n-13-8H_{\frac{n-1}{2}}%
\end{align*}
Now, we have to prove $6n-13-8H_{\frac{n-1}{2}}>0,$ which is again done by
induction. This time left side of inequality is denoted by $F(n).$ We have
$F(5)=5>0.$ Supposing $F(n)>0$ for a $n\geq5,$ we obtain
\[
F(n+2)=F(n)+\frac{12n-4}{n+1}>0.
\]

\end{proof}

\section{Maximal unicyclic graphs}

We will find maximal graphs by introducing transformations of graph $G$ to
$G^{\prime}$ which \textbf{increase} the value of $H_{A},$ therefore
$H_{A}(G)<H_{A}(G^{\prime}),$ i.e. $H_{A}(G^{\prime})-H_{A}(G)>0.$ Since
$H_{A}(G^{\prime})-H_{A}(G)$ is a sum over all pairs of vertices in $G,$ it is
very convenient to introduce notation $\Delta(u,v)$ for the contribution of a
pair $u,v\in G$ to the sum $H_{A}(G^{\prime})-H_{A}(G)$. The problem to solve
will be negative contributions of certain pairs of vertices, for which we will
have to find enough pairs with positive contribution to compensate in the sum.

\begin{lemma}
\label{lema1_1iznimkaCnOdd}For odd $n\geq5$ it holds that $H_{A}(C_{n}%
)<H_{A}(CS_{3,n-3}).$
\end{lemma}

\begin{proof}
For odd $n,$ it is easily verified that $H_{A}(C_{n})=4n\cdot H_{\frac{n-1}%
{2}}.$ From Lemma \ref{lema0_valueExtremal} we have $H_{A}(CS_{3,n-3}%
)=\frac{3}{2}(n^{2}-n+2).$ Now the claim follows from Lemma
\ref{lema0_pomocnaHarmonic}.
\end{proof}

\begin{figure}[h]
\begin{center}
\includegraphics[scale=0.2]{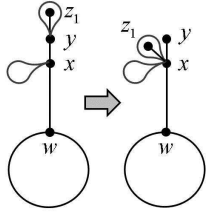}
\end{center}
\caption{Graph transformation from the proof of Lemma
\ref{lema1_2arbitraryToCyclestar}.}%
\label{Figure05}%
\end{figure}

\begin{lemma}
\label{lema1_2arbitraryToCyclestar}For every $G\in\mathcal{U}(n)$ which is not
a cycle-star graph, there is a cycle-star graph $G^{\prime}\in\mathcal{U}(n)$
such that $H_{A}(G)<H_{A}(G^{\prime}).$
\end{lemma}

\begin{proof}
Let $C$ be the only cycle in $G$ and let $w\in C$ be the branching vertex. Let
$T$ be connected component of $G\backslash C$ containing $w$. Note that $T$
must be a tree. Let $z_{1}$ be a leaf in $T$ furthest from $w.$ Let $P=w\ldots
xyz_{1}$ be the shortest path between $w$ and $z_{1}.$ Let $z_{1},z_{2}%
,\ldots,z_{k}$ be all leafs neighboring to $y$. Since $z_{1}$ is a leaf
furthest from $w,$ it follows that $k=\delta_{G}(y)-1$. Let $G^{\prime}$ be a
graph obtained from $G$ by deleting edges $z_{i}y$ for $i=1,\ldots,k$ and
adding edges $z_{i}x$ instead. This transformation is illustrated in Figure
\ref{Figure05}. We have to prove that $H_{A}(G)<H_{A}(G^{\prime}).$ Because of
the definition of the index $H_{A},$ the problem are pairs of vertices whose
distance increases and vertices whose degree decreases. Note that in this
transformation the only pairs $u,v$ whose distance increases are pairs $u=y$
and $v=z_{i}.$ Also, the only vertex for which degree decreases is $y.$
Therefore, contributions $\Delta(y,z_{i})$ will be negative due to increase in
distance, while contributions $\Delta(y,v)$, where $v\in G,$ will (possibly)
be negative due to decrease in degree of vertex $y.$ Note that contributions
$\Delta(y,v)$ are not necessarily negative, since degree of $v$ can increase
or the distance $d_{G}(y,v)$ can also decrease.

Let us first consider the problem with increase in distance. We will show that
the negative contribution $\Delta(y,z_{i})$ is compensated by positive
contribution $\Delta(x,z_{i}).$ More formally, we have
\begin{align*}
\Delta(y,z_{i})+\Delta(x,z_{i})  &  =\frac{1+1}{2}-\frac{1+k+1}{1}%
+\frac{\delta_{G}(x)+k+1}{1}-\frac{\delta_{G}(x)+1}{2}=\\
&  =\frac{\delta_{G}(x)-1}{1}-\frac{\delta_{G}(x)-1}{2}>0.
\end{align*}
Let us now consider the problem with degrees. We have to consider
contributions $\Delta(y,v)$, where $v\in G.$ First, note that we have already
considered and compensated contributions $\Delta(y,z_{i}).$ Further, note that
not all of these contributions are necessarily negative, since there can be
increase in degree of the other vertex. So, for $v=x$ we have
\[
\Delta(y,x)=\frac{1+(k+\delta(x))}{1}-\frac{(k+1)+\delta_{G}(x)}{1}=0.
\]
For $v\in G\backslash\{x,y,z_{1},\ldots,z_{k}\}$ negative contribution
$\Delta(y,v)$ can be compensated with positive contribution of $\Delta(x,v). $
More formally, we have
\begin{align*}
\Delta(y,v)+\Delta(x,v)  &  =\frac{1+\delta_{G}(x)}{d_{G}(y,v)}-\frac
{k+1+\delta_{G}(x)}{d_{G}(y,v)}+\frac{k+\delta_{G}(x)+\delta_{G}(v)}%
{d_{G}(y,v)-1}-\frac{\delta_{G}(x)+\delta_{G}(v)}{d_{G}(y,v)-1}=\\
&  =\frac{k}{d_{G}(y,v)-1}-\frac{k}{d_{G}(y,v)}>0.
\end{align*}
Therefore, we have proved $H_{A}(G)<H_{A}(G^{\prime}).$ If $G^{\prime}$ is
cycle-star graph the proof is over, else this transformation must be repeated
finitely many times to obtain cycle-star graph and then the proof is over.
\end{proof}

\begin{figure}[h]
\begin{center}%
\begin{tabular}
[c]{llll}%
a) & $\text{\raisebox{-1\height}{\includegraphics[scale=0.2]{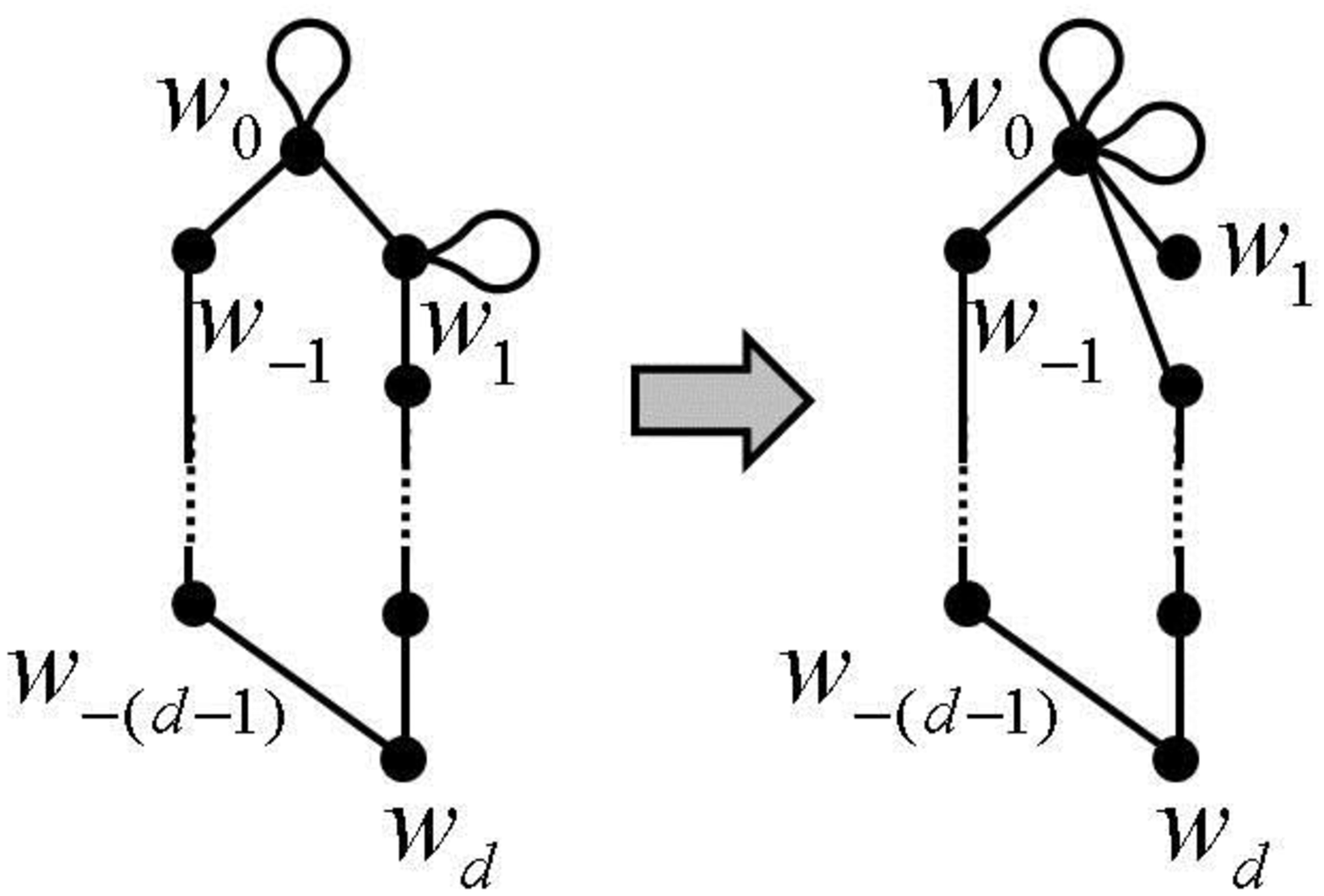}}}$
& b) & $\text{\raisebox{-1\height}{\includegraphics[scale=0.2]{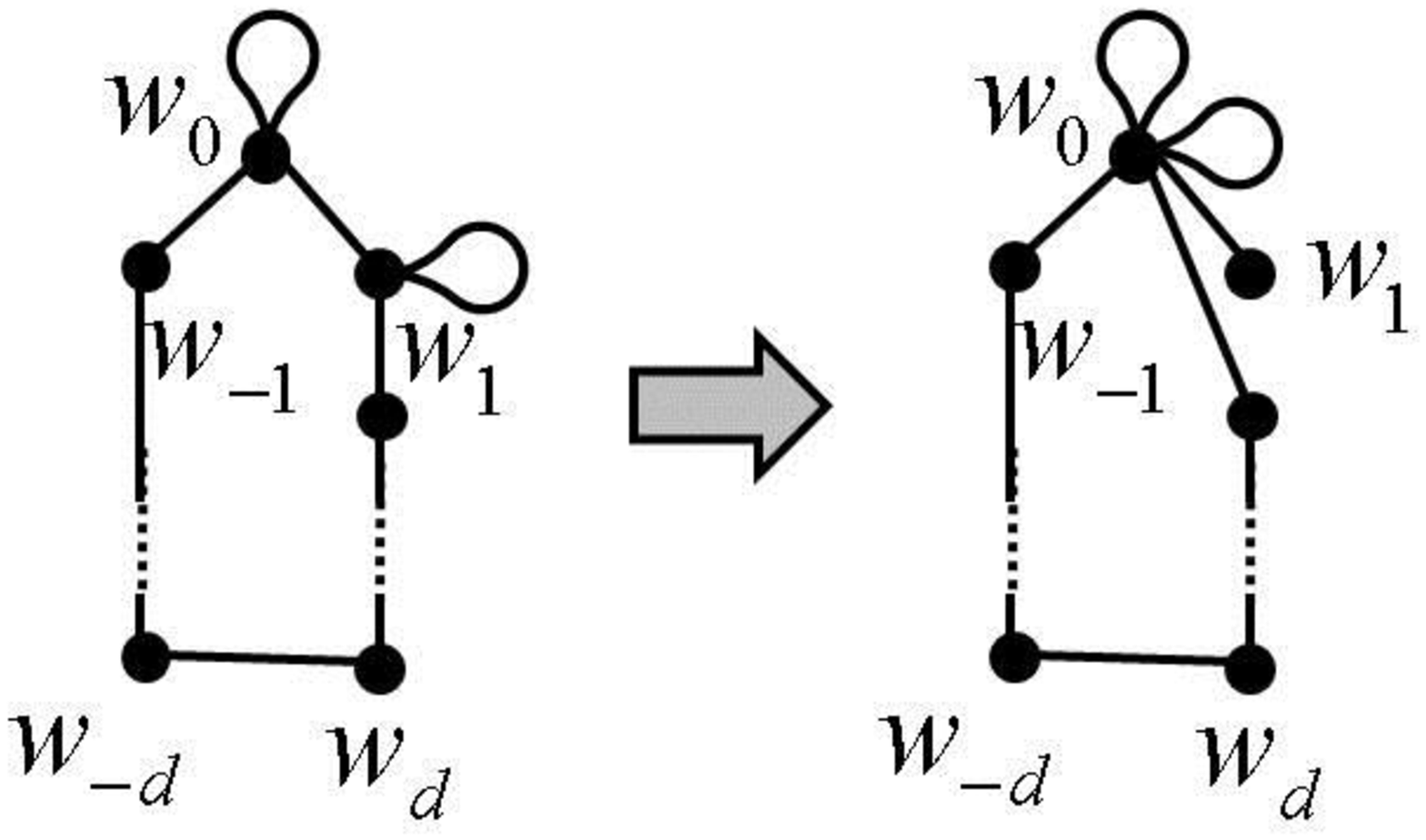}}}%
$%
\end{tabular}
\end{center}
\caption{Graph transformation from the proof of Lemma
\ref{lema1_3decreasingCycle}: a) even $c,$ b) odd $c.$}%
\label{Figure06}%
\end{figure}

\begin{lemma}
\label{lema1_3decreasingCycle}Let $G\in\mathcal{U}(n)$ be a cycle-star with
the length of the only cycle being $c\geq4.$ Then there is a cycle-star
$G^{\prime}\in\mathcal{U}(n)$ with the length of the only cycle being $3$ for
which $H_{A}(G)<H_{A}(G^{\prime}).$
\end{lemma}

\begin{proof}
Let $C$ be the only cycle in $G$ and let $c$ be the length of $C.$ Let
$d=\left\lfloor c/2\right\rfloor .$ In the case of even $c$ we will denote
vertices on $C$ by $C=w_{-(d-1)}\ldots w_{d},$ in the case of odd $c$ we will
denote $C=w_{-d}\ldots w_{d}.$ Without loss of generality we may assume that
$w_{-d}$ (if it exists, since it exists only for odd $c$) is of minimum degree
among vertices on $C$. It is convenient to introduce the notation
$k_{i}=\delta_{G}(w_{i}).$ Now, leafs neighboring to $w_{i}$ wil be denoted by
$x_{i,j}$ ($j=1,\ldots,k_{i}-2$) and we define $V_{i}=\{w_{i},x_{i,j}%
:j=1,\ldots,k_{i}-2\}.$ Now, let $V^{+}=\left(  V_{1}\cup\ldots\cup
V_{d}\right)  \backslash\{w_{1}\}$ and $V^{-}=\left(  V_{0}\cup\ldots\cup
V_{-(d-1)}\right)  \backslash\{w_{0}\}.$ Obviously, in the case of even $c$ it
holds that $V=\{w_{0},w_{1}\}\cup V^{+}\cup V^{-},$ while in the case of odd
$c$ it holds that $V=\{w_{0},w_{1}\}\cup V^{+}\cup V^{-}\cup V_{-d}.$

Now, let $G^{\prime}$ be a graph obtained from $G$ by deleting all edges
$vw_{1}$ incident to $w_{1},$ except $w_{0}w_{1},$ and add the edge $vw_{0}$
instead. This transformation is illustrated in Figure \ref{Figure06}.
Obviously, $G^{\prime}$ is a cycle-star graph in which the only cycle is of
the length $c-1.$ We have to prove that $H_{A}(G)<H_{A}(G^{\prime}).$ Note
that in this transformation the only pairs $u,v$ whose distance increases are
pairs $u=w_{1}$ and $v\in V^{+}\cup V_{-d}$ (recall that $V_{-d}$ only exists
for odd $c$). Also, the only vertex for which degree decreases is $w_{1}.$
Therefore, we have to consider all contributions $\Delta(w_{1},v)$ where $v\in
G.$

If $v=w_{0}$ we have $\Delta(w_{1},w_{0})=0,$ else if $v\in V^{+}$ we have
\[
\Delta(w_{1},v)+\Delta(w_{0},v)=\frac{k_{0}-1}{d_{G}(w_{1},v)}-\frac{k_{0}%
-1}{d_{G}(w_{1},v)+1}>0,
\]
else if $v\in V^{-}$ we have
\[
\Delta(w_{1},v)+\Delta(w_{0},v)=\frac{k_{1}-1}{d_{G}(w_{1},v)-1}-\frac
{k_{1}-1}{d_{G}(w_{1},v)}>0.
\]
Note that in the case of even $c$ this completes the proof. In the case of odd
$c$ we still have to consider contributions $\Delta(w_{1},v)$ where $v\in
V_{-d}.$ First, note that for $G=C_{n}$ the claim of this lemma follows from
Lemma \ref{lema1_1iznimkaCnOdd}. So, we will suppose $G\not =C_{n}.$ Now, for
negative contributions $\Delta(w_{1},v)$ where $v\in V_{-d}$, positive
contribution $\Delta(w_{0},v)$ will not suffice to compensate, so we will have
to find more pairs with positive contribution for compensation. To prove this
more formally, we will distinguish cases $v=w_{-d}\in V_{-d}$ and
$v=x_{-d,j}\in V_{-d}.$

If $v=w_{-d}\in V_{-d}$ we have
\[
\Delta(w_{1},w_{-d})+\Delta(w_{0},w_{-d})=\frac{1+k_{-d}}{d+1}-\frac{1+k_{-d}%
}{d}<0.
\]
Therefore, we have to find more pairs with positive contributions to
compensate. Since $G\not =C_{n}$ there has to be at least one leaf $x_{i,1}$
in $G.$ Without loss of generality we can assume $i\leq0$. Then for $x_{i,1} $
and $w_{j}$ where $j=d-i$, recalling that $w_{-d}$ is of minimum degree on $C$
(i.e. $k_{j}\geq k_{-d}$)$,$ we have
\[
\Delta(x_{i,1},w_{j})=\frac{1+k_{j}}{d}-\frac{1+k_{j}}{d+1}\geq\lbrack
k_{j}\geq k_{-d}]\geq\frac{1+k_{-d}}{d}-\frac{1+k_{-d}}{d+1}.
\]
So, obviously $\Delta(w_{1},w_{-d})+\Delta(w_{0},w_{-d})+\Delta(x_{i,1}%
,w_{j})\geq0.$

Finally, if $v=x_{-d,j}\in V_{-d}$ we have
\[
\Delta(w_{1},x_{-d,i})+\Delta(w_{0},x_{-d,i})=\frac{2}{d+2}-\frac{2}{d+1}<0.
\]
Again, it follows that we have to find more pairs with positive contribution
to compensate. Since $w_{-d}$ is of minimum degree on $C,$ it follows that
every $w_{i}$ has at least as many leafs as $w_{-d}.$ Let us now consider pair
of leafs $x_{0,1}$ and $x_{d,i}.$ We have
\[
\Delta(x_{0,1},x_{d,i})=\frac{2}{d+1}-\frac{2}{d+2}>0.
\]
Obviously, we have $\Delta(w_{1},x_{-d,i})+\Delta(w_{0},x_{-d,i}%
)+\Delta(x_{0,1},x_{d,i})=0$ which completes the proof.
\end{proof}

\begin{figure}[h]
\begin{center}
\includegraphics[scale=0.2]{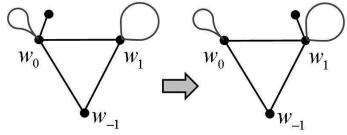}
\end{center}
\caption{Graph transformation from the proof of Lemma
\ref{lema1_4leafsOfTriangle}.}%
\label{Figure08}%
\end{figure}

\begin{lemma}
\label{lema1_4leafsOfTriangle}Let $G\in\mathcal{U}(n),$ where $G\not =%
CS_{3,n-3}$, be a cycle-star graph with the length of the only cycle being
$c=3.$ Then $H_{A}(G)<H_{A}(CS_{3,n-3}).$
\end{lemma}

\begin{proof}
Let $C=w_{-1}w_{0}w_{1}$ be a cycle in $G$ and let $k_{i}=\delta_{G}(w_{i}). $
Let $x_{i,j}$ be all leafs attached to $w_{i}$ and let $V_{i}=\{w_{i}%
,x_{i,j}:j=1,\ldots,k_{i}-2\}.$ Since $G\not =CS_{3,n-3},$ it follows that at
least two vertices on $C$ are branching. Without loss of generality we can
suppose that $w_{0}$ has minimum degree and $w_{1}$ maximum degree among
branching vertices on $C.$ Let $G^{\prime}$ be a graph obtained from $G$ by
deleting the edge $x_{0,1}w_{0}$ and adding the edge $x_{0,1}w_{1}$ instead.
We have to prove $H_{A}(G)<H_{A}(G^{\prime}).$

Note that in this transformation the only pairs $u,v$ whose distance increases
are pairs $u=x_{0,1}$ and $v\in V_{0}.$ The only vertex for which degree
decreases is $w_{0}.$ Let us first consider the problem with the increase in
distances. It is easily verified that $\Delta(x_{0,1},x_{0,j})+\Delta
(x_{0,1},x_{1,j})=0.$ Also, we have
\[
\Delta(x_{0,1},w_{0})+\Delta(x_{0,1},w_{1})=\frac{k_{1}+1-k_{0}}{1}%
-\frac{k_{1}+1-k_{0}}{2}>[k_{1}\geq k_{0}]>0.
\]
Let us now consider the problem with the increase in degree. It is easily
verified that for $v\in V_{-1}$ we have $\Delta(w_{0},v)+\Delta(w_{1},v)=0$
and also that $\Delta(w_{0},w_{1})=0.$ Further, we have
\begin{align*}
\Delta(w_{0},x_{0,j\geq2})+\Delta(w_{1},x_{0,j\geq2})  &  =\frac{k_{1}%
+1-k_{0}}{1}-\frac{k_{1}+1-k_{0}}{2}>[k_{1}\geq k_{0}]>0.\\
\Delta(w_{0},x_{1,j})+\Delta(w_{1},x_{1,j})  &  =\frac{3}{2}>0.
\end{align*}

\end{proof}

\begin{theorem}
\label{tm1_maximalGraph}Let $G\in\mathcal{U}(n).$ Then
\[
H_{A}(G)\leq\frac{3}{2}(n^{2}-n+2)
\]
with equality if and only if $G=CS_{3,n-3.}$
\end{theorem}

\begin{proof}
Using Lemmas \ref{lema1_2arbitraryToCyclestar}, \ref{lema1_3decreasingCycle}
and \ref{lema1_4leafsOfTriangle} we first transform $G$ to a cycle-star graph,
then we reduce the length of the cycle to $3$, so that we can finally
transform it to $C_{3,n-3.}$ In each transformation the value of $H_{A}$
increases, so $G=C_{3,n-3.}$ is the only extremal unicyclic graph. Note that
the case of $G=C_{n}$ is covered by Lemma \ref{lema1_1iznimkaCnOdd} for $n$
odd, while for even $n$ it is covered by Lemma \ref{lema1_3decreasingCycle}.
Now the bound for $H_{A}$ follows from Lemma \ref{lema0_valueExtremal}.
\end{proof}

\section{Minimal unicyclic graphs}

As in previous section, we will find maximal graphs by introducing
transformation of graph $G$ to $G^{\prime},$ but which now \textbf{decrease}
the value of $H_{A}.$ Therefore, in this section $\Delta(u,v)$ will denote the
contribution of a pair of vertices $u,v\in G$ to the sum $H_{A}(G)-H_{A}%
(G^{\prime}).$ The problem will again be pairs of vertices with negative
contribution, for which we will have to find enough pairs with positive
contribution to compensate in the sum.

\begin{figure}[h]
\begin{center}
\includegraphics[scale=0.2]{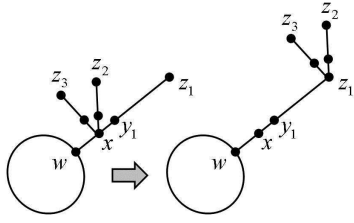}
\end{center}
\caption{Graph transformation from the proof of Lemma
\ref{lema2_1subtreesToPaths}.}%
\label{Figure09}%
\end{figure}

\begin{lemma}
\label{lema2_1subtreesToPaths}For every $G\in\mathcal{U}(n)$ which is not a
cycle-path graph, there is a cycle-path graph $G^{\prime}\in\mathcal{U}(n)$
such that $H_{A}(G)>H_{A}(G^{\prime}).$
\end{lemma}

\begin{proof}
Let $w$ be a branching vertex on $C,$ let $T$ be the connected component of
$G\backslash C$ containing $w,$ let $x\in T$ be a branching vertex in $T$
furthest from $w.$ Let $P_{i}=xy_{i}\ldots z_{i}$ be all paths in $T$ starting
in $x$ such that $d(w,x)<d(w,y_{i})$ for every $i.$ Since $x$ is the branching
vertex furthest from $w,$ these paths are obviously vertex disjoint except for
the vertex $x.$ Without loss of generality we may assume that $P_{1}$ is the
longest among these paths. Let $G^{\prime}$ be a graph obtained from $G$ by
deleting edges $xy_{i}$ for $i=2,\ldots,\delta_{G}(x)-1,$ and adding edges
$z_{1}y_{i}$ instead. This transformation is illustrated in Figure
\ref{Figure09}. Note that $G^{\prime}$ is a unicyclic graph. We have to prove
that $H_{A}(G)>H_{A}(G^{\prime}).$ Since the value of the index $H_{A}$ has to
decrease, the problem are the pairs of vertices whose distance decreases or
vertices whose degree increases. Distances possibly decrease only for pairs
$u,v$ where $u\in P_{i}$ ($i=2,\ldots,\delta_{G}(x)-1$) and $v\in P_{1}.$ The
only vertex for which the degree increases in this transformation is $z_{1}.$

Let us first consider the problem with distances. Let $v\mapsto v^{\prime}$ be
the automorphism of the path $P_{1}$ such that $x^{\prime}=z_{1}.$ For $u\in
P_{i}$ we have
\[
\Delta(u,x)+\Delta(u,z_{1})=\frac{1}{d_{G}(u,x)}-\frac{1}{d_{G}%
(u,x)+\left\vert P_{1}\right\vert }>0.
\]
For $u\in P_{i}$ and $v\in P_{1}\backslash\{x,z_{1}\},$ if $v=v^{\prime}$ then
$d_{G}(u,v)=d_{G^{\prime}}(u,v)$ so $\Delta(u,v)=0,$else if $v\not =v^{\prime
}$ then we have $d_{G}(x,v)=d_{G}(z,v^{\prime})$ and $d_{G}(z,v)=d_{G}%
(x,v^{\prime})$ so it is easily verified that $\Delta(u,v)+\Delta(u,v^{\prime
})=0.$

Let us now consider the problem with degree. For $v\in P_{i}$ we have already
considered and compensated negative contributions $\Delta(z_{1},v).$ Let again
$v\mapsto v^{\prime}$ be the automorphism of path $P_{1}$ such that
$x^{\prime}=z_{1}.$ If we now consider $v\in P_{1},$ then for $v=x$ we have
\[
\Delta(z_{1},x)=\frac{1+\delta_{G}(x)}{d_{G}(z_{1},x)}-\frac{1+\delta
_{G}(x)-2+2}{d_{G}(z_{1},x)}=0,
\]
while for $v\in P_{1}\backslash\{z_{1},x\}$, we have
\[
\Delta(z_{1},v)+\Delta(x,v^{\prime})=\frac{\delta_{G}(x)-2}{d_{G}(x,v^{\prime
})}-\frac{\delta_{G}(x)-2}{d_{G}(z_{1},v)}=\left[  d_{G}(z_{1},v)=d_{G}%
(x,v^{\prime})\right]  =0.
\]
Finally, we have to consider $v\in G\backslash(\cup_{i}P_{i}).$ We have
\[
\Delta(z_{1},v)+\Delta(x,v)=\frac{\delta_{G}(x)-2}{d_{G}(x,v)}-\frac
{\delta_{G}(x)-2}{d_{G}(z_{1},v)}>\left[  d_{G}(x,v)<d_{G}(z_{1},v)\right]
>0.
\]
Therefore, we have proved that $H_{A}(G)>H_{A}(G^{\prime}).$ If $G^{\prime}$
is a cycle-path graph, then the proof is completed. If not, then by repeating
this transformation finitely many times we obtain a cycle path graph
$G^{\prime}$ for which $H_{A}(G)>H_{A}(G^{\prime})$, so the proof is complete.
\end{proof}

\begin{figure}[h]
\begin{center}%
\begin{tabular}
[c]{llll}%
a) & $\text{\raisebox{-1\height}{\includegraphics[scale=0.2]{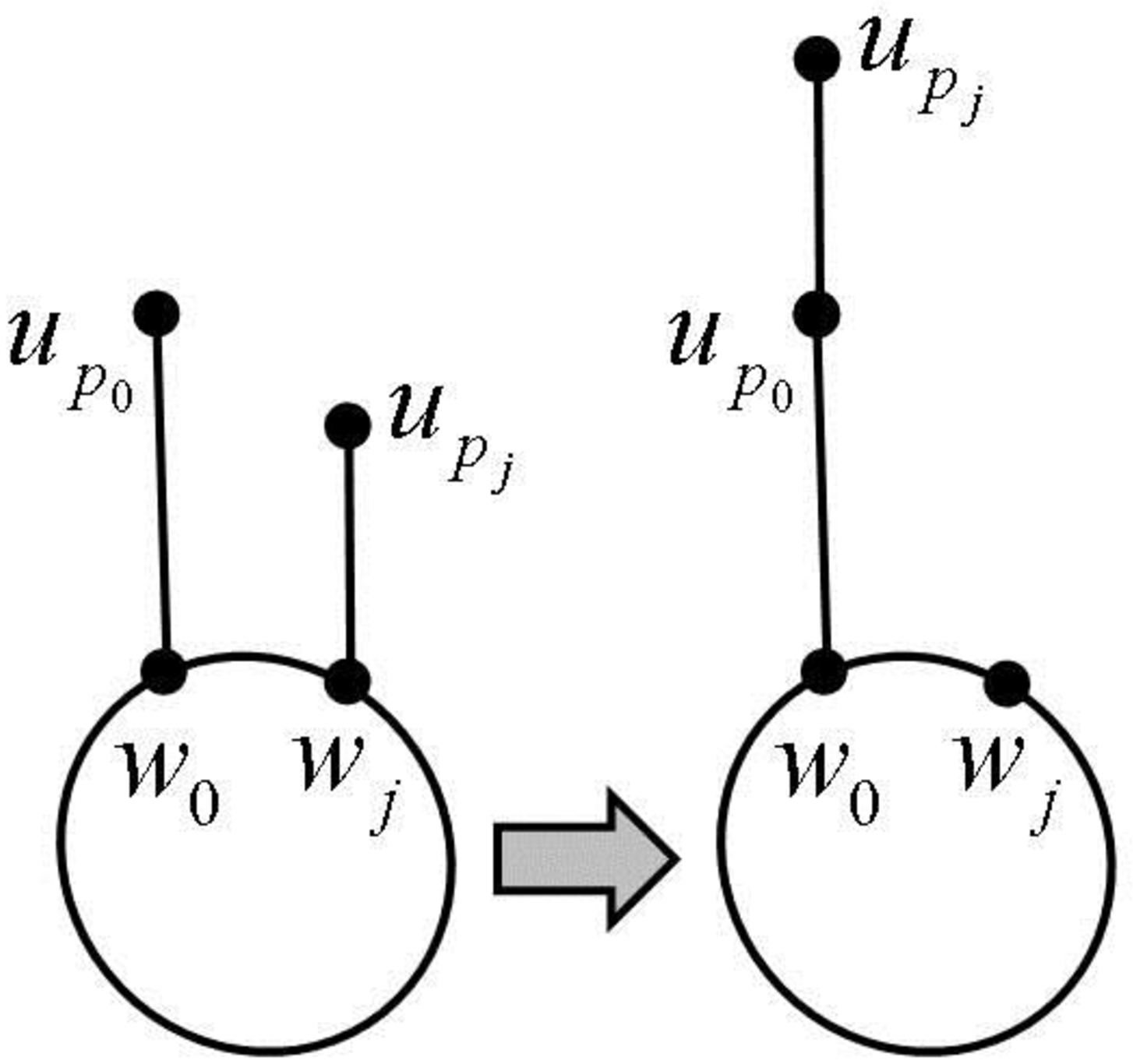}}}$
& b) & $\text{\raisebox{-1\height}{\includegraphics[scale=0.2]{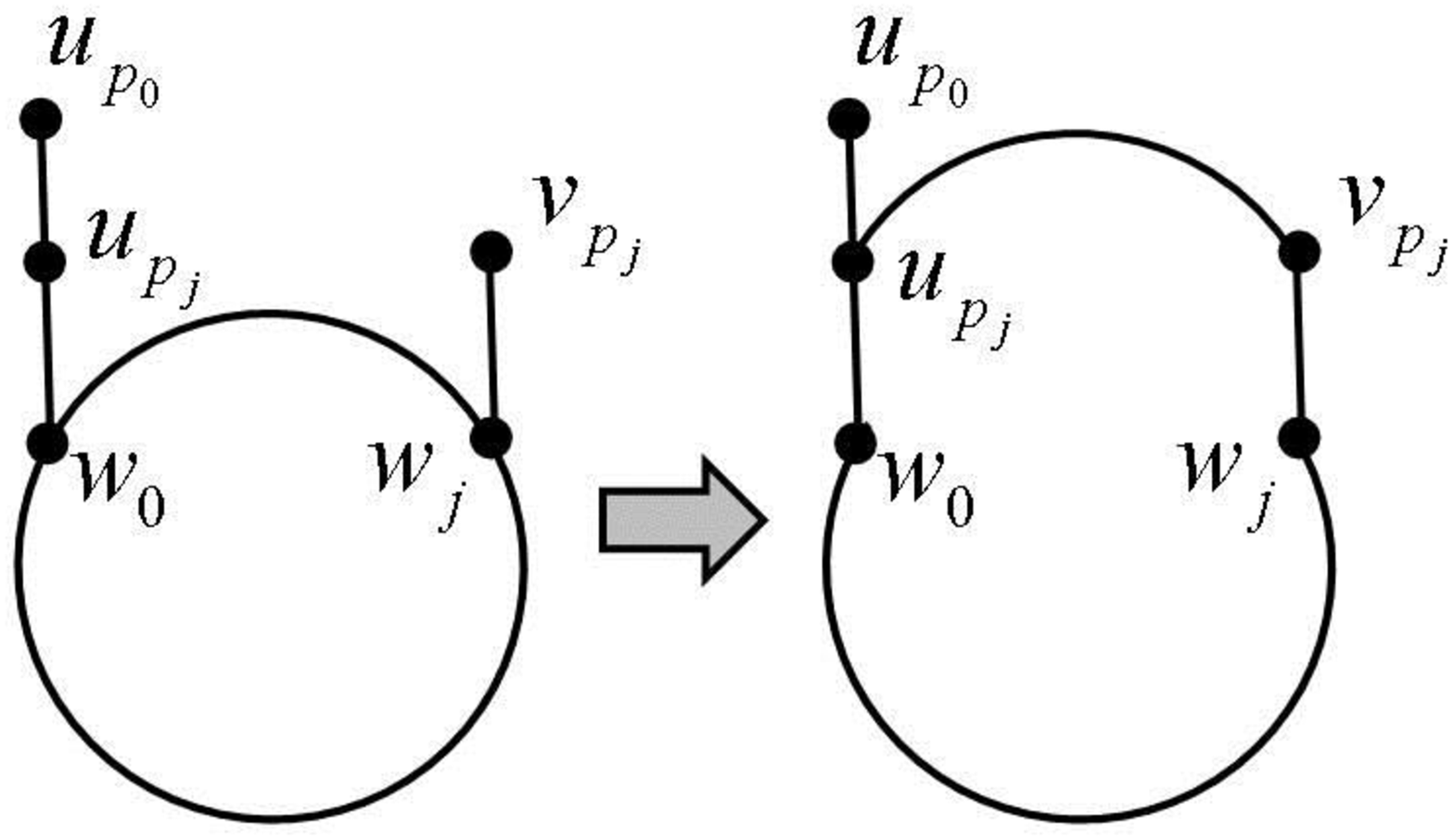}}}%
$%
\end{tabular}
\end{center}
\caption{Graph transformation from the proof of Lemma
\ref{lema2_2decreasingBranches}: a) case 1, b) case 2a.}%
\label{Figure10}%
\end{figure}

\begin{lemma}
\label{lema2_2decreasingBranches}Let $G\in\mathcal{U}(n)$ be a cycle-path
graph with at least $2$ branching vertices. Then there a cycle-path graph
$G^{\prime}\in\mathcal{U}(n)$ with only one branching vertex such that
$H_{A}(G)>H_{A}(G^{\prime}).$
\end{lemma}

\begin{proof}
Let $C$ be a cycle in $G$ with vertices denoted by $w_{i}.$ Let $P_{i}$ be a
path appended to a branching vertex $w_{i}\in C$ (here $P_{i}$ includes
$w_{i}$) and $p_{i}=\left\vert P_{i}\right\vert .$ We distinguish several cases.

CASE 1. There are two consecutive branching vertices $w_{i},w_{j}$ on $C$ such
that $d(w_{i},w_{j})\leq p_{i}$ or $d(w_{i},w_{j})\leq p_{j}.$ Without loss of
generality we may assume that $0=i<j\leq\left\lfloor \frac{c}{2}\right\rfloor
$, $d(w_{0},w_{j})\leq p_{0}$ and $p_{0}\geq p_{j}.$ Let $P_{0}=w_{0}%
u_{1}\ldots u_{p_{0}}$ and $P_{j}=w_{j}v_{1}\ldots v_{p_{j}}.$ Let $G^{\prime
}$ be a graph obtained from $G$ by deleting the edge $w_{j}v_{1}$ and adding
the edge $u_{p_{0}}v_{1}.$ This transformation is illustrated in Figure
\ref{Figure10} a). Note that $G^{\prime}$ is cycle-path graph with one branch
less than $G.$ We have to prove $H_{A}(G)>H_{A}(G^{\prime}).$ Let $P_{A}$ be
the shortest path in $G$ connecting vertices $u_{p_{0}}$ and $v_{p_{j}}.$ Note
that distances possibly decrease only for pairs of vertices $u,v\in P_{A}.$
The only vertex for which the degree increases is $u_{p_{0}}.$

Let us first consider the problem with distances. Let $P_{A}^{\prime}$ be a
path in $G^{\prime}$ connecting vertices $v_{p_{j}}$ and $w_{j}.$ Let
$v\mapsto v^{\prime}$ be an isomorphism of paths $P_{A}$ and $P_{A}^{\prime} $
such that $v_{p_{j}}^{\prime}=v_{p_{j}}.$ Note that $u_{p_{0}}^{\prime}=w_{j}$
(and vice versa) and also $w_{0}^{\prime}\in P_{0}$ since $d(w_{0},w_{j})\leq
p_{0}.$ Now, we first consider $u=v_{a}\in P_{j}.$ If $v=v_{b}\in P_{j}$ we
have $\Delta(v_{a},v_{b})=0,$ else if $v=w_{j}$ and $v^{\prime}=u_{p_{0}}$ we
have
\[
\Delta(v_{a},w_{j})+\Delta(v_{a},u_{p_{0}})=\frac{1}{d_{G}(v_{a},w_{j})}%
-\frac{1}{d_{G}(v_{a},w_{j})+p_{0}+j}>0,
\]
else if $v=v^{\prime}$ from $d_{G^{\prime}}(v_{a},v)=d_{G}(v_{a},v)$ we have
$\Delta(v_{a},v)=0,$ else if $v=w_{0}$ and $v^{\prime}=w_{0}^{\prime} $ we
have
\[
\Delta(v_{a},w_{0})+\Delta(v_{a},w_{0}^{\prime})=\frac{1}{d_{G}(v_{a}%
,w_{j})+j}-\frac{1}{d_{G}(v_{a},w_{j})+p_{0}}\geq\lbrack j\leq p_{0}]\geq0.
\]
else if $v\not =v^{\prime}$ from $d_{G}(w_{j},v^{\prime})=d_{G}(u_{p_{0}},v) $
and $d_{G}(u_{p_{0}},v^{\prime})=d_{G}(w_{j},v)$ we have $\Delta
(v_{a},v)+\Delta(v_{a},v^{\prime})=0.$ Therefore we have covered all pairs
$u,v\in P_{A}$ for which $u\in P_{j}.$ Let now $u=u_{p_{0}}.$ We have already
considered pairs with $v\in P_{j}\backslash\{w_{j}\}$ in previous text. If
$v=w_{j}$ we have $\Delta(u_{p_{0}},w_{j})=0,$ else
\[
\Delta(u_{p_{0}},v)+\Delta(w_{j},v^{\prime})=\frac{1}{d_{G}(w_{j},v^{\prime}%
)}-\frac{1}{d_{G}(u_{p_{0}},v)}=[d_{G}(w_{j},v^{\prime})=d_{G}(u_{p_{0}%
},v)]=0
\]
Finally, for $u,v\in P_{A}\backslash(P_{j}\cup\{u_{p_{0}},w_{j}\})$ both the
degrees and the distances are not changed in the transformation, so we have
$\Delta(u,v)=0.$

Therefore, we have considered and compensated all negative contributions due
to decrease in distances. Let us now consider all negative contributions due
to increase in degree of $u_{p_{0}}.$ Note that we have already considered and
compensated pairs $u_{p_{0}},v$ where $v\in P_{A}$. Let now $v\in G\backslash
P_{A}.$ We have
\[
\Delta(u_{p_{0}},v)+\Delta(w_{j},v)=\frac{1}{d_{G}(w_{j},v)}-\frac{1}%
{d_{G}(u_{p_{0}},v)}=[d_{G}(w_{j},v)\leq d_{G}(u_{p_{0}},v)]\geq0.
\]

CASE 2. For every two consecutive branching vertices $w_{i},w_{j}$ on $C$ it
holds that $d(w_{i},w_{j})>p_{i}$ and $d(w_{i},w_{j})>p_{j}.$

SUBCASE 2a. There are at least $4$ branching vertices on $C.$

Let $w_{i}$ and $w_{j}$ be the pair of branching vertices on minimum distance.
Without loss of generality we may assume that $0=i<j\leq\left\lfloor \frac
{c}{4}\right\rfloor .$ Let $P_{0}=w_{0}u_{1}\ldots u_{p_{0}}$ and $P_{j}%
=w_{j}v_{1}\ldots v_{p_{j}}$ be paths appended to $w_{0} $ and $w_{j}.$ The
condition of Case 2 is now read as $j>p_{0}$ and $j>p_{j}. $ Without loss of
generality we may assume $p_{0}\geq p_{j}.$ Now, let $G^{\prime}$ be a graph
obtained from $G$ by deleting edges $w_{0}w_{1}$ and $w_{j-1}w_{j}$ and adding
edges $u_{p_{j}}w_{1}$ and $w_{j-1}v_{p_{j}}$ instead. This transformation is
illustrated in Figure \ref{Figure10} b).\ Graph $G^{\prime}$ is obviously
cycle-path graph with one branch less than $G.$ We have to prove
$H_{A}(G)>H_{A}(G^{\prime}).$ Let $P_{A}$ be the shortest path in $G$
connecting vertices $u_{p_{0}}$ and $v_{p_{j}}.$ Note that distances possibly
decrease only for pairs of vertices $u,v\in P_{A}.$ The only vertices for
which degree increases are $u_{p_{j}}$ and $v_{p_{j}}.$

Let us first consider the problem with distances. Let $P_{B}$ be the shortest
path in $G$ connecting vertices $u_{p_{j}}$ and $v_{p_{j}}$ ($P_{B}$ is
subpath of $P_{A}$). Let $P_{B}^{\prime}$ be the shortest path in $G^{\prime}$
connecting vertices $w_{0}$ and $w_{j}.$ It is important to note that because
of the condition of subcase ($2p_{j}+j<3\left\lfloor \frac{c}{4}\right\rfloor
$) shortest path between $w_{0}$ and $w_{j}$ both before and after the
transformation goes through the same side of cycle, i.e. paths $P_{B}$ and
$P_{B}^{\prime}$ are of the same length and contain the same vertices (though
not in the same order). Let $v\mapsto v^{\prime}$ be the isomorphism of paths
$P_{B}$ and $P_{B}^{\prime}$ such that $u_{p_{j}}^{\prime}=w_{0}.$ We first
consider pairs $u,v\in P_{B}.$ For $u,v\in\{u_{p_{j}},v_{p_{j}},w_{0},w_{j}\}$
it is easily verified that $\Delta(u_{p_{j}},w_{j})=\Delta(v_{p_{j}}%
,w_{0})=\Delta(u_{p_{j}},w_{0})=\Delta(v_{p_{j}},w_{j})=0$ and%
\[
\Delta(u_{p_{j}},v_{p_{j}})+\Delta(w_{0},w_{j})=\frac{3-\delta_{G}(u_{p_{j}}%
)}{j}-\frac{3-\delta_{G}(u_{p_{j}})}{2p_{j}+j}>0.
\]
Now, let us consider pairs $u,v\in P_{B}$ where $u\in P_{0}.$ First, let
$u=u_{p_{j}}\in P_{0}.$ Then, for $b<p_{j}$ we have
\begin{align*}
\Delta(u_{p_{j}},u_{b})+\Delta(w_{0},u_{b}^{\prime})  &  =\frac{1}{d_{G}%
(w_{0},u_{b}^{\prime})}-\frac{1}{d_{G}(u_{p_{j}},u_{b})}=\left[
d_{G}(u_{p_{j}},u_{b})=d_{G}(w_{0},u_{b}^{\prime})\right]  =0,\\
\Delta(u_{p_{j}},v_{b})+\Delta(w_{0},v_{b}^{\prime})  &  =\frac{2-\delta
_{G}(u_{p_{j}})}{j+p_{j}-b}-\frac{2-\delta_{G}(u_{p_{j}})}{p_{j}+j+b}>0,
\end{align*}
and for $0<b<j$ we have
\[
\Delta(u_{p_{j}},w_{b})+\Delta(w_{0},w_{b})=\frac{2-\delta_{G}(u_{p_{j}})}%
{b}-\frac{2-\delta_{G}(u_{p_{j}})}{p_{j}+b}>0.
\]
Now, let $u=u_{a}\in P_{0}$ for $a<p_{j}.$ Then for $b<p_{j}$, if $u_{a}%
=u_{a}^{\prime}$ and $v_{b}=v_{b}^{\prime}$ (i.e. $a=b=\frac{p_{j}}{2} $) we
have $\Delta(u_{a},v_{b})=0,$ else we have $\Delta(u_{a},v_{b})+\Delta
(u_{a}^{\prime},v_{b}^{\prime})=0.$ Also, for $0<b<j$, if $u_{a}=u_{a}%
^{\prime}$ (i.e. $a=\frac{p_{j}}{2}$) then $\Delta(u_{a},w_{b})=0,$ else
$\Delta(u_{a},w_{b})+\Delta(u_{a}^{\prime},w_{b})=0.$ Completely analogously
one can obtain analogous results for pairs $u,v$ where $u\in P_{j}.$ Now, for
all unconsidered pairs of vertices $u,v\in P_{B}$ neither degrees change, nor
the distance, therefore for those pair it holds that $\Delta(u,v)=0.$ Hence,
we have considered all pairs $u,v\in P_{B}.$

It remains to consider pairs $u,v\in P_{A}$ where $u\in P_{A}\backslash
P_{B}.$ Now, let $u=u_{p_{j}+a}\in P_{A}\backslash P_{B}.$ If $v\in
P_{A}\backslash P_{B}$ then obviously $\Delta(u,v)=0,$ else if $v\in
\{w_{0},u_{p_{j}}\}$ we have
\begin{align*}
\Delta(u_{p_{j}+a},w_{0})+\Delta(u_{p_{j}+a},u_{p_{j}})  &  =\frac{1}{p_{j}%
+a}-\frac{1}{a}<0,\\
\Delta(w_{-a},w_{0})+\Delta(w_{-a},u_{p_{j}})  &  =\frac{1}{a}-\frac
{1}{a+p_{j}}>0,
\end{align*}
so we obtain $\Delta(u_{p_{j}+a},w_{0})+\Delta(u_{p_{j}+a},u_{p_{j}}%
)+\Delta(w_{-a},w_{0})+\Delta(w_{-a},u_{p_{j}})=0.$ Else if $v\in
\{w_{j},v_{p_{j}}\}$ we have
\begin{align*}
\Delta(u_{p_{j}+a},w_{j})+\Delta(u_{p_{j}+a},v_{p_{j}})  &  =\frac
{1}{j+a+p_{j}}-\left(  \frac{\delta_{G}(u_{p_{j}+a})+2}{a+j}-\frac{\delta
_{G}(u_{p_{j}+a})+1}{p_{j}+a+j+p_{j}}\right)  ,\\
\Delta(w_{j+a},w_{0})+\Delta(w_{j+a},u_{p_{j}})  &  =\left(  \frac{5}%
{a+j}-\frac{4}{j+a+p_{j}+p_{j}}\right)  -\frac{1}{j+a+p_{j}},
\end{align*}
so we obtain
\[
\Delta(u_{p_{j}+a},w_{j})+\Delta(u_{p_{j}+a},v_{p_{j}})+\Delta(w_{j+a}%
,w_{0})+\Delta(w_{j+a},u_{p_{j}})=\frac{3-\delta_{G}(u_{p_{j}+a})}{a+j}%
-\frac{3-\delta_{G}(u_{p_{j}+a})}{j+a+p_{j}+p_{j}}>0.
\]
Else if $v\in\{u_{0},\ldots,u_{p_{j}-1}\}$ we have $\Delta(u_{p_{j}+a},v)=0.$
Else if $v\in\{w_{1},\ldots,w_{j-1}\}$ we have
\[
\Delta(u_{p_{j}+a},w_{b})+\Delta(w_{-a},w_{b})=\frac{2-\delta_{G}(u_{p_{j}%
+a})}{a+b}-\frac{2-\delta_{G}(u_{p_{j}+a})}{a+b+p_{j}}>0.
\]
Else if $v\in\{v_{1},\ldots,v_{p_{j}-1}\}$ we have
\[
\Delta(u_{p_{j}+a},v_{b})+\Delta(w_{-a},v_{b}^{\prime})=\frac{2-\delta
_{G}(u_{p_{j}+a})}{a+j+p_{j}-b}-\frac{2-\delta_{G}(u_{p_{j}+a})}{a+p_{j}%
+j+b}>0.
\]

Therefore, we have considered and compensated all negative contributions due
to decrease in distances. Let us now consider all negative contributions due
to increase in degree of $u_{p_{j}}$ and $v_{p_{j}}.$ First, note that we have
already considered pairs $u_{p_{j}},v$ and $v_{p_{j}},v$ for which $v\in
P_{A}.$ For $u=u_{p_{j}}$ we have also already considered pairs where
$v=w_{-a}$ or $v=w_{j+a}$ ($a=1,\ldots,p_{0}-p_{j}$). Now, let $v$ be an
unconsidered vertex, then from $d_{G}(w_{0},v)\leq d_{G^{\prime}}(w_{0},v)$
and $d_{G}(u_{p_{j}},v)=d_{G^{\prime}}(u_{p_{j}},v)$ we have
\[
\Delta(u_{p_{j}},v)+\Delta(w_{0},v)\geq\frac{1}{d_{G}(w_{0},v)}-\frac{1}%
{d_{G}(u_{p_{j}},v)}>\left[  d_{G}(w_{0},v)<d_{G}(u_{p_{j}},v)\right]  >0.
\]
Completely analogously one obtains $\Delta(v_{p_{j}},v)+\Delta(w_{j},v)>0$ and
the proof of this subcase is over.

SUBCASE\ 2b. There are exactly $3$ branching vertices on $C.$

Let $w_{0},w_{j}$ and $w_{k}$ be three branching vertices on $C.$ Let us
denote $d_{1}=d_{G}(w_{0},w_{j}),$ $d_{2}=d_{G}(w_{j},w_{k}),$ $d_{3}%
=d_{G}(w_{k},w_{0}).$ Without loss of generality we may assume that $d_{3}%
\geq\max\{d_{1},d_{2}\}.$ We will also need $d_{4}=\min\{d_{1}+d_{2}%
,d_{3}\}\geq d_{2}.$ Let $P_{0}=w_{0}u_{1}\ldots u_{p_{0}},$ $P_{j}=w_{j}%
v_{1}\ldots v_{p_{j}},$ $P_{k}=w_{k}z_{1}\ldots z_{p_{k}}$ be paths appended
to branching vertices on $C$. Without loss of generality we may assume
$p_{0}\leq p_{k}.$ Let $G^{\prime}$ be the graph obtained from $G $ by
deleting edge $w_{j}v_{1}$ and adding the edge $u_{p_{0}}v_{1}.$ Note that
$G^{\prime}$ is a cycle-path graph with one branch less than $G.$ We have to
prove $H_{A}(G)>H_{A}(G^{\prime}).$ Note that distances possibly decrease only
for pairs $u,v$ for which $u\in P_{j}$ and $v\in G.$ The only vertex whose
degree increases is $u_{p_{0}}.$

We will first consider the problem with distances, so let $u=v_{a}\in P_{j}.$
Let $P_{A}$ be the shortest path in $G$ connecting vertices $v_{p_{j}}$ and
$u_{p_{0}},$ let $P_{A}^{\prime}$ be the shortest path in $G^{\prime}$
connecting vertices $v_{p_{j}}$ and $w_{j}.$ Let $v\mapsto v^{\prime}$ be an
isomorphism of paths $P_{A}$ and $P_{A}^{\prime}$ such that $v_{p_{j}%
}=v_{p_{j}}^{\prime}.$ Note that $w_{0}^{\prime}\in\{w_{1},\ldots,w_{j-1}\}$
since $p_{0}<j$ (supposition of Case 2). First, we will consider cases where
$v\in P_{A}.$ If $v=v_{b}\in P_{j},$ then we have $\Delta(v_{a},v_{b})=0.$
Else if $v\in\{w_{j},u_{p_{0}},w_{0}^{\prime},w_{0}\}$ we have%
\begin{align*}
\Delta(v_{a},u_{p_{0}})+\Delta(v_{a},w_{j})  &  =\frac{1}{a}-\frac{1}%
{a+d_{1}+p_{0}}>0,\\
\Delta(v_{a},w_{0})+\Delta(v_{a},w_{0}^{\prime})  &  =\frac{1}{a+d_{1}}%
-\frac{1}{a+p_{0}}<\left[  d_{1}>p_{0}\right]  <0,
\end{align*}
so obviously
\[
\Delta(v_{a},u_{p_{0}})+\Delta(v_{a},w_{j})+\Delta(v_{a},w_{0})+\Delta
(v_{a},w_{0}^{\prime})=\frac{1}{a}-\frac{1}{a+p_{0}}+\frac{1}{a+d_{1}}%
-\frac{1}{a+d_{1}+p_{0}}>0.
\]
Else if $v=v^{\prime}$ from $d_{G}(w_{j},v)=d_{G}(u_{p_{0}},v)$ we have
$\Delta(v_{a},v)=0.$ Else if $v\not =v^{\prime}$ from $d_{G}(v_{a},v^{\prime
})=d_{G^{\prime}}(v_{a},v)$ and $d_{G^{\prime}}(v_{a},v^{\prime})=d_{G}%
(v_{a},v)$ we have $\Delta(v_{a},v)+\Delta(v_{a},v^{\prime})=0.$ Therefore, we
have considered all cases where $v\in P_{A}.$ Let us now consider cases where
$v\in C\backslash P_{A}.$ Let $v\mapsto v^{\prime}$ be an automorphism of
cycle $C$ such that $w_{0}^{\prime}=w_{j}$ and $w_{j}^{\prime}=w_{0}.$ If
$v=v^{\prime}$ from $d_{G}(w_{j},v)=d_{G}(u_{p_{0}},v)$ we have%
\[
\Delta(v_{a},v)=\frac{\delta_{G}(v_{a})+\delta_{G}(v)}{d_{G}(v_{a}%
,w_{j})+d_{G}(w_{j},v)}-\frac{\delta_{G}(v_{a})+\delta_{G}(v)}{d_{G}%
(v_{a},w_{j})+p_{0}+d_{G}(w_{j},v)}>0
\]
else if $v\in\{w_{k},w_{k}^{\prime}\}$ we have
\begin{align*}
\Delta(v_{a},w_{k})+\Delta(v_{a},w_{k}^{\prime})  &  =\frac{\delta_{G}%
(v_{a})+3}{a+d_{2}}-\frac{\delta_{G}(v_{a})+3}{a+p_{0}+d_{4}}+\frac{\delta
_{G}(v_{a})+2}{a+d_{4}}-\frac{\delta_{G}(v_{a})+2}{a+p_{0}+d_{2}}>\\
&  >\left[  \frac{\delta_{G}(v_{a})+3}{a+d_{4}}-\frac{\delta_{G}(v_{a}%
)+3}{a+p_{0}+d_{4}}>0\right]  >\\
&  >\frac{\delta_{G}(v_{a})+3}{a+d_{2}}-\frac{\delta_{G}(v_{a})+2}%
{a+p_{0}+d_{2}}-\frac{1}{a+d_{4}}\geq\left[  d_{4}\geq d_{2}\right] \\
&  \geq\frac{\delta_{G}(v_{a})+2}{a+d_{2}}-\frac{\delta_{G}(v_{a})+2}%
{a+p_{0}+d_{2}}>0,
\end{align*}
else from $d_{G}(w_{j},v^{\prime})=d_{G}(w_{0},v)$ and $d_{G}(w_{0},v^{\prime
})=d_{G}(w_{j},v)$ we have
\[
\Delta(v_{a},v)+\Delta(v_{a},v^{\prime})=\frac{\delta_{G}(v_{a})+2}%
{a+d_{G}(w_{0},v)}-\frac{\delta_{G}(v_{a})+2}{a+p_{0}+d_{G}(w_{0},v)}%
+\frac{\delta_{G}(v_{a})+2}{a+d_{G}(w_{j},v)}-\frac{\delta_{G}(v_{a}%
)+2}{a+p_{0}+d_{G}(w_{j},v)}>0.
\]
Finally, we have to consider $v=z_{b}\in P_{k}.$ We have%
\[
\Delta(v_{a},z_{b})=\frac{\delta_{G}(v_{a})+\delta_{G}(z_{b})}{a+d_{2}%
+b}-\frac{\delta_{G}(v_{a})+\delta_{G}(z_{b})}{a+d_{4}+b}\geq\lbrack d_{2}\leq
d_{4}]\geq0.
\]
Therefore, we have considered and compensated all negative contributions due
to decrease in distances.

Let us now consider all negative contributions due to increase in degree of
$u_{p_{0}},$ so let $u=u_{p_{0}}.$ We have already considered pairs $u,v$
where $v\in P_{j}\backslash\{w_{j}\}.$ For the remaining possibilities for
$v,$ we will again use already introduced isomorphisms of $P_{A}$ and $C$. If
$v=w_{j}$ then $\Delta(u_{p_{0}},w_{j})=0,$ else if $v\in P_{A}\backslash
P_{j}$ we have%
\[
\Delta(u_{p_{0}},v)+\Delta(w_{j},v^{\prime})=\frac{1}{d_{G}(w_{j},v^{\prime}%
)}-\frac{1}{d_{G}(u_{p_{0}},v)}=\left[  d_{G}(u_{p_{0}},v)=d_{G}%
(w_{j},v^{\prime})\right]  =0,
\]
else if $v\in C\backslash P_{A}$, we have
\[
\Delta(u_{p_{0}},v)+\Delta(w_{j},v^{\prime})=\frac{1}{d_{G}(w_{0},v^{\prime}%
)}-\frac{1}{d_{G}(u_{p_{0}},v)}=\left[  d_{G}(w_{0},v^{\prime})=d_{G}%
(u_{p_{0}},v)\right]  =0,
\]
else if $v=z_{a}\in P_{k}\backslash\{w_{k}\}$ we have%
\[
\Delta(u_{p_{0}},z_{a})+\Delta(w_{j},z_{a})=\frac{1}{d_{2}+a}-\frac{1}%
{p_{0}+d_{4}+a}>0.
\]

SUBCASE 2c. There are exactly $2$ branching vertices on $C.$ Let $w_{0}$ and
$w_{j}$ be branching vertices on $C$ and let $P_{0}=w_{0}u_{1}\ldots u_{p_{0}%
},$ $P_{j}=w_{j}v_{1}\ldots v_{p_{j}}$ be paths appended to branching
vertices. Let $G^{\prime}$ be a graph obtained from $G$ by deleting the edge
$w_{j}v_{1}$ and adding the edge $u_{p_{0}}v_{1}$ instead. Graph $G^{\prime}$
is obviously cycle-path graph with only one branch. Proof that $H_{A}%
(G)>H_{A}(G^{\prime})$ is completely analogous to the proof of subcase 2b, one
just doesn't have to consider vertex $w_{k}$ separately and there are no
vertices $z_{k}.$

So, in all cases we have proved $H_{A}(G)>H_{A}(G^{\prime}).$ Since in all
cases $G^{\prime}$ is a cycle-path graph with one branch less than in $G,$ we
have either obtained cycle path $G^{\prime}$ which has only one branch, or by
repeating the transformation finitely many times we will obtain such graph.
Therefore, the lemma is proved.
\end{proof}

\begin{figure}[h]
\begin{center}
{\includegraphics[scale=0.2]{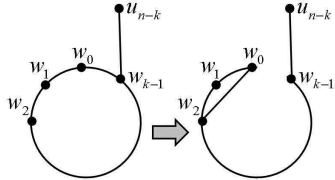}}
\end{center}
\caption{Graph transformation from the proof of Lemma
\ref{lema2_3increasingTail}.}%
\label{Figure12}%
\end{figure}

\begin{lemma}
\label{lema2_3increasingTail}Let $G=C_{k,n-k}$ where $k\geq4$ and $n\geq5.$
Then $H_{A}(G)>H_{A}(C_{3,n-3}).$
\end{lemma}

\begin{proof}
Let us denote vertices in $G$ so that for the only cycle in $G$ holds
$C=w_{0}w_{1}\ldots w_{k-1}$ where $4\leq k\leq n.$ If there is a branching
vertex in $G,$ without loss of generality we may assume it is $w_{k-1}$ and
$P_{k-1}=w_{k-1}u_{1}\ldots u_{n-k}$ is the only path appended to $w_{k-1}.$
Let $G^{\prime}$ be a graph obtained from $G$ by deleting edge $w_{0}w_{k-1} $
and adding the edge $w_{0}w_{2}$ instead. This transformation is illustrated
in Figure \ref{Figure12}. Note that $G^{\prime}=CP_{3,n-3}.$ We have to prove
$H_{A}(G)>H_{A}(G^{\prime}).$ Note that distances in this transformation
decrease only for $u,v$ where $u=w_{0}$ and $v=w_{i}$ ($i=2,\ldots
,\left\lfloor k/2\right\rfloor $). The only vertex whose degree increases is
$w_{2}.$ For the ease of proving the lemma, we introduce $d_{m}=\min\{3,k-3\}$
and the path $P_{A}=w_{2}w_{3}\ldots w_{k-1}$ in $G$ with the automorphism
$v\mapsto v^{\prime}$ of $P_{A}$ such that $w_{2}^{\prime}=w_{k-1}.$ Now, we
distinguish two cases with respect to whether $G=C_{n}$ ($k=n$) or
$G=G=C_{k,n-k}$ ($4\leq k\leq n-1$).

CASE 1. Let $G=C_{n}$ ($k=n$). We will first consider the problem with
distances. Let $u=w_{0}.$ If $v=w_{2}$ we have
\[
\Delta(w_{0},w_{2})+\Delta(w_{0},w_{k-1})=\frac{1}{2}\frac{7k-20}{k-2}>0,
\]
else if $v=w_{i}$ ($i=3,\ldots,\left\lfloor k/2\right\rfloor $) we have
\[
\Delta(w_{0},w_{i})+\Delta(w_{0},w_{k-i+1})=\frac{4}{i}-\frac{4}{k-i}%
\geq\lbrack i\leq k-i]\geq0.
\]
Let us now consider the problem with the increase in degree of $w_{2},$ so let
$u=w_{2}.$ Note that we have already considered pairs $u,v$ where $v=w_{0}.$
If $v=w_{1}$ we have
\[
\Delta(w_{2},w_{1})+\Delta(w_{k-1},w_{1})=\frac{k-5}{k-2}\geq\lbrack
k=n\geq5]\geq0,
\]
else if $v=w_{k-1}$ we have
\[
\Delta(w_{2},w_{k-1})=\frac{2+2}{d_{m}}-\frac{3+1}{k-3}\geq\lbrack k-3\leq
d_{m}]\geq0,
\]
else using the automorphism of $P_{A}$ (and supposing $d_{P_{A}}(w_{2},v)\leq
d_{P_{A}}(w_{2},v^{\prime})$) from $d_{G}(w_{2},v)=d_{G}(w_{k-1},v^{\prime})$
we have $\Delta(w_{2},v)+\Delta(w_{k-1},v^{\prime})=0.$

CASE 2. Let $G=C_{k,n-k}.$ Again, we first consider the 'problem' with
distances. Let $u=w_{0},$ we have to consider $v=w_{2}$ and $v\in
\{w_{3},\ldots,w_{\left\lfloor k/2\right\rfloor }\}.$ We have
\begin{align*}
\Delta(w_{0},w_{2})+\Delta(w_{0},w_{k-1})  &  =\frac{2k-8}{k-2}\geq\lbrack
k\geq4]\geq0\\
\Delta(w_{0},w_{i})+\Delta(w_{0},w_{k-i+1})  &  =\frac{4}{i}-\frac{4}{k-i}%
\geq\lbrack i\leq k-i]\geq0.
\end{align*}
Now, let us consider the problem with the increase in degree of $w_{2},$ so
let $u\in w_{2}.$ Note that we have already considered $v=w_{0}.$ We have to
consider $v=w_{1},$ $v=w_{k-1},$ $v\in P_{A}\backslash\{w_{1},w_{k-1}\},$
$v=u_{a}\in P_{k-1}.$ We have
\begin{align*}
\Delta(w_{2},w_{1})+\Delta(w_{k-1},w_{1})  &  =\frac{1}{2}\frac{3k-14}%
{k-2}>0\text{ for }k\geq5,\\
\Delta(w_{2},w_{k-1})  &  =\frac{2+3}{d_{m}}-\frac{3+2}{k-3}\geq\lbrack
d_{m}\leq k-3]\geq0.
\end{align*}
Further, assuming $d_{P_{A}}(w_{v},v)\leq d_{P_{A}}(w_{v},v^{\prime})$ we have
$\Delta(w_{2},v)+\Delta(w_{k-1},v^{\prime})=0.$ Finally, using $d_{m}\leq k-3$
we obtain
\[
\Delta(w_{2},u_{a})+\Delta(w_{k-1},u_{a})\geq\frac{1}{a}-\frac{1}{d_{m}+a}>0.
\]
Therefore, the only problem is $\Delta(w_{2},w_{1})+\Delta(w_{k-1},w_{1})$ for
$k=4.$ But note that in that case $d_{m}=\min\{4,1\}=1,$ so we have
\[
\Delta(w_{2},w_{1})+\Delta(w_{k-1},w_{1})+\Delta(w_{2},u_{1})+\Delta
(w_{k-1},u_{1})=\frac{1}{2}\frac{12-14}{4-2}+\frac{1}{1}-\frac{1}{1+1}=0.
\]
Note that in this case we have not proved strict inequality if there is only
one $u_{a},$ i.e. if $n=5.$ But in that case it is easily verified that
$H_{A}(C_{4,1})>H_{A}(C_{3,2}).$ The positive contribution which makes the
difference is $\Delta(w_{0},u_{a})$, but which was not considered in the proof.
\end{proof}

Note that $C_{4}$ and $C_{3,1}$ are the only unicyclic graphs on $n=4$
vertices. It holds that%
\[
H_{A}(C_{4})=20<21=H_{A}(C_{3,1}).
\]
So, for $n=4$ graph $C_{4}$ is the only minimal unicyclic graph, while for
$n\geq5$ the answer to the question of minimal unicyclic graph is given by the
following theorem.

\begin{theorem}
\label{tm2_minimalGraph}Let $G\in\mathcal{U}(n)$ for $n\geq5.$ Then
\[
H_{A}(G)\geq4%
{\displaystyle\sum_{i=1}^{n-2}}
H_{n-i-1}+H_{n-3}+3H_{n-2}+\frac{6n-13}{n-2}%
\]
with equality if and only if $G=CP_{3,n-3}.$
\end{theorem}

\begin{proof}
Using Lemmas \ref{lema2_1subtreesToPaths}, \ref{lema2_2decreasingBranches} and
\ref{lema2_3increasingTail} we first transform a unicyclic graph to cycle-path
graph, then we decrease the number of branches in obtained cycle-path graph,
so that finally we can transform it to $CP_{3,n-3}.$ In each of these
transformations the value of $H_{A}$ strictly decreases, so $CP_{3,n-3}$ is
the only extremal graph. Now, the bound follows from Lemma
\ref{lema0_valueExtremal}. Note that the case of $C_{n}$ is covered by Lemma
\ref{lema2_3increasingTail}.
\end{proof}

\section{Conclusion}

In this paper we defined cycle-star graph $CS_{k,n-k}$ to be a graph
consisting of cycle of length $k$ and $n-k$ leafs appended to the same vertex
of the cycle. Also, we defined cycle-path graph to be a graph consisting of
cycle of length $k$ and of path on $n-k$ vertices whose one end is linked to a
vertex on a cycle. We establish that $CS_{3,n-3}$ is the only maximal
unicyclic graph (see Theorem \ref{tm1_maximalGraph}), while $CP_{3,n-3}$ is
the only minimal unicyclic graph (see Theorem \ref{tm2_minimalGraph}), with
respect to additively weighted Harary index. The values of additively weighted
Harary index of $CS_{3,n-3}$ and $CP_{3,n-3}$ are established in Lemma
\ref{lema0_valueExtremal}, so these values are the upper and the lower bound
for the value of Harary index on the class of unicyclic graphs. For further
research it would be interesting to investigate the values of Harary index on
classes of graphs with given parameters, the relation of this variant of
Harary index with other topological indices and similar.

\section{Acknowledgements}

The support of the EUROCORES Programme EUROGIGA (project GReGAS) of the
European Science Foundation is gratefully acknowledged.


\bigskip

\begin{thebibliography}{99}                                                                                               %


\bibitem {ref_Doslic(2013)}Y. Alizadeh, A. Iranmanesh, T. Do\v{s}li\'{c},
Additively weighted Harary index of some composite graphs, Discret. Math.
313:1 (2013) 26-34.

\bibitem {ref_Bruckler(2011)}F.M. Br\"{u}ckler, T. Do\v{s}li\'{c}, A. Graovac,
I. Gutman, On a class of distance-based molecular structure descriptors, Chem.
Phys. Lett. 503 (2011) 336--338.

\bibitem {ref_Dobrynin(2001)}A.A. Dobrynin, R. Entringer, I. Gutman, Wiener
index of trees: theory and applications, Acta Appl. Math. 76 (2001) 211--249.

\bibitem {ref_Dobrynin(2002)}A.A. Dobrynin, I. Gutman, S. Klav\v{z}ar, P.
\v{Z}igert, Wiener index of hexagonal systems, Acta Appl. Math. 72 (2002) 247--294.

\bibitem {ref_SvojstvaIlic}A. Ili\'{c}, G. Yu, L. Feng, The Harary index of
trees, arXiv:1104.0920v3 [math.CO].

\bibitem {ref_Ivanciuc(1993)}O. Ivanciuc, T.S. Balaban, A.T. Balaban, Design
of topological indices, part 4, reciprocal distance matrix, related local
vertex invariants and topological indices, J. Math. Chem. 12 (1993) 309--318.

\bibitem {ref_Nikolic(1995)}S. Nikoli\'{c}, N. Trinajsti\'{c}, Z. Mihali\'{c},
The Wiener index: development and applications, Croat. Chem. Acta 68 (1995) 105--129.

\bibitem {ref_Olavsic(1993)}D. Plav\v{s}i\'{c}, S. Nikoli\'{c}, N.
Trinajsti\'{c}, Z. Mihali\'{c}, On the Harary index for the characterization
of chemical graphs, J. Math. Chem. 12 (1993) 235--250.

\bibitem {ref_SvojstvaWagner}S. Wagner, H. Wang, X.D. Zhang, Distance-based
graph invariants of trees and the Harary index, Filomat 27:1 (2013), 41--50.

\bibitem {ref_Wiener(1947)}H. Wiener, Structural determination of the paraffin
boiling points, J. Amer. Chem. Soc. 69 (1947) 17--20.

\bibitem {ref_HararyUnicyclic}K. Xu, K. Ch. Das, Extremal Unicyclic and
Bicyclic Graphs with Respect to Harary Index, B. Malays. Math. Sci. So. 36
(2013) 373-383.

\bibitem {ref_Zhou(2008)}B. Zhou, X. Cai, N. Trinajsti\'{c}, On Harary index,
J. Math. Chem. 44 (2008) 611--618.
\end{thebibliography}
\end{document}